\documentclass[leqno]{article}
\usepackage{geometry}
\usepackage{graphicx}	
\usepackage[cp1251]{inputenc}
\usepackage[english]{babel}
\usepackage{mathtools}
\usepackage{amsfonts,amssymb,mathrsfs,amscd,amsmath,amsthm}

\usepackage[dvipdfmx]{hyperref}

\usepackage{verbatim}

\usepackage{url}

\usepackage{stmaryrd}
\usepackage{color}

\def\ig#1#2#3#4{\begin{figure}[!ht]\begin{center}%
\includegraphics[height=#2\textheight]{#1.eps}\caption{#4}\label{#3}%
\end{center}\end{figure}}

\def\thtext#1{
  \catcode`@=11
  \gdef\@thmcountersep{. #1}
  \catcode`@=12
}

\def\threst{
  \catcode`@=11
  \gdef\@thmcountersep{.}
  \catcode`@=12
}

\theoremstyle{plain}

\newtheorem{thm}{Theorem}[section]
\newtheorem{prop}[thm]{Proposition}
\newtheorem{cor}[thm]{Corollary}

\newtheorem{lem}[thm]{Lemma}

\theoremstyle{definition}
\newtheorem{dfn}[thm]{Definition}
\newtheorem{rk}[thm]{Remark}

\newtheorem{examp}[thm]{Example}

 \catcode`@=11
 \def\.{.\spacefactor\@m}
 \catcode`@=12

\def\C{{\mathbb C}}
\def\N{{\mathbb N}}

\def\R{\mathbb R}

\def\a{\alpha}

\def\e{\varepsilon}

\def\g{\gamma}

\def\l{\lambda}

\def\v{\varphi}

\def\0{\emptyset}
\def\:{\colon}
\def\<{\langle}
\def\>{\rangle}
\def\[{\llbracket}
\def\]{\rrbracket}

\def\rom#1{\emph{#1}}
\def\({\rom(}
\def\){\rom)}

\def\ss{\subset}

\def\toGH{\xrightarrow{\operatorname{GH}}}
\def\sp{\supset}

\def\x{\times}

\newcommand{\overbar}[1]{\mkern 1.5mu\overline{\mkern-1.5mu#1\mkern-1.5mu}\mkern 1.5mu}

\def\bA{{\bar A}}

\def\bX{\overbar{X}}

\def\diam{\operatorname{diam}}

\def\dis{\operatorname{dis}}

\def\id{\operatorname{id}}

\def\Lip{\operatorname{Lip}}

\def\cH{{\cal H}}

\def\cP{{\cal P}}

\def\cR{{\cal R}}

\def\tx{{\tilde x}}

\begin{document}
\title{Gromov--Hausdorff Distance Between Segment and Circle}
\author{Yibo Ji, Alexey A. Tuzhilin}
\date{}
\maketitle

\begin{abstract}
We calculate the Gromov--Hausdorff distance between a line segment and a circle in the Euclidean plane. To do that, we introduced a few new notions like round spaces and nonlinearity degree of a metric space.
\end{abstract}

\section{Introduction}
\markright{\thesection.~Introduction}
The Gromov--Hausdorff distance measures the difference between any two metric spaces. There are a few possibilities to do that. One of them, to embed isometrically the both spaces into all possible other metric spaces, then the least possible Hausdorff distance between the images will be just that characteristic. Another way is to establish correspondences between these spaces and to measure the least possible distortion of the spaces metrics produced by these correspondences.

The Gromov--Hausdorff distance has many applications, for example, it helps to investigate the growth of groups, or it can be used in image recognition. However, to get concrete values of the distance, even in the cases of ``simple spaces'' like line segment and the standard circle, is a very non-trivial task. In this paper we discuss our original technique enabled to obtain the exact values in the latter case.

\section{Preliminaries}
\markright{\thesection.~Preliminaries}

In what follows, we work with various non-empty metric spaces, and the distance between points $p$ and $q$ of a metric space we denote by $|pq|$, independently on the choice of the space. For a metric space $X$, $x\in X$, non-empty $A,B\ss X$, $r>0$, and $s\ge0$, we use the following notions and notations:
\begin{itemize}
\item $U_r(x)=\bigl\{y\in X:|xy|<r\bigr\}$ (\emph{open ball with center $x$ and radius $r$\/});
\item $B_s(x)=\bigl\{y\in X:|xy|\le s\bigr\}$ (\emph{closed ball with center $x$ and radius $s$\/});
\item $|xA|=\inf_{a\in A}|xa|$ (\emph{distance between $x$ and $A$\/});
\item $U_r(A)=\bigl\{y\in X:|yA|<r\bigr\}$ (\emph{open $r$-neighborhood of $A$\/});
\item $B_s(A)=\bigl\{y\in X:|yA|\le s\bigr\}$ (\emph{closed $s$-neighborhood of $A$\/});
\item $d_H(A,B)=\inf\bigl\{t:A\ss U_t(B)\text{ and }U_t(A)\sp B\bigr\}$ (\emph{Hausdorff distance between $A$ and $B$\/}).
\end{itemize}

For any non-empty set $X$, we denote by $\cP_0(X)$ the set of all non-empty subsets of $X$. If $X$ is a metric space, then the Hausdorff distance on $\cP_0(X)$ satisfies the triangle inequality, but can be equal $\infty$, and can vanish for different subsets. However, if we restrict $d_H$ to the set $\cH(X)$ of all non-empty closed bounded subsets of $X$, then $d_H$ becomes a metric, see~\cite{BBI}.

\begin{thm}[\cite{BBI}]\label{thm:HausdorfProps}
Given a metric space $X$, the following properties are simultaneously presented or not in the both $X$ and $\cH(X)$\rom: completeness, total boundness, compactness.
\end{thm}

Given $r>0$, a subset of a metric space is called \emph{$r$-separated\/} if the distance between any its different points is at least $r$.

Given two sets $X$ and $Y$, \emph{a correspondence between $X$ and $Y$} is each subset $R\ss X\x Y$ such that for any $x\in X$ there exists $y\in Y$ with $(x,y)\in R$ and, vise versa, for any $y\in Y$ there exists $x\in X$ with $(x,y)\in R$. Let $\cR(X,Y)$ denote the set of all correspondences between $X$ and $Y$. If $X$ and $Y$ are metric spaces, and $R\in\cR(X,Y)$, then we define \emph{the distortion $\dis R$ of $R$} as follows:
$$
\dis R=\sup\Bigl\{\bigl||x_1x_2|-|y_1y_2|\bigr|:(x_1,y_1),\,(x_2,y_2)\in R\Bigr\}.
$$
Further, \emph{the Gromov--Hausdorff distance $d_{GH}(X,Y)$ between metric spaces $X$ and $Y$}, or for short \emph{GH-distance}, is the value
$$
d_{GH}(X,Y)=\frac12\inf\bigl\{\dis R:R\in\cR(X,Y)\bigr\}.
$$
An equivalent definition: it is the infimum of Hausdorff distances $d_H(X',Y')$ between all possible subsets $X'$ and $Y'$ of the metric spaces $Z$, provided $X'$ is isometric to $X$, and $Y'$ is isometric to $Y$. It is well-known~\cite{BBI} that $d_{GH}$ is a metric on the set of isometry classes of compact metric spaces, in particular, two compact metric spaces are isometric if and only if the GH-distance between them vanishes. In general situation, $d_{GH}$ satisfies the triangle inequality, is bounded for bounded spaces (for non-bounded spaces it can be infinite), and can vanish for non-isometric spaces. Also, for any $A,B\in\cP_0(X)$, it holds $d_{GH}(A,B)\le d_H(A,B)$. In particular, if $\bA$ is the closure of $A$, then $d_{GH}(A,\bA)=d_H(A,\bA)=0$.

If $X_1,X_2,\ldots$ and $X$ are some metric spaces such that $d_{GH}(X_i,X)\to0$, then we say that the sequence $X_i$ is \emph{Gromov--Hausdorff convergent\/} to $X$ and write $X_i\toGH X$.

It is easy to see that for any correspondence $R\in\cR(X,Y)$, its closure in $X\x Y$ has the same distortion, thus, to achieve $d_{GH}(X,Y)$, it suffices to consider only closed correspondences. In other words, if $\cR_c(X,Y)$ is the set of all closed correspondences between $X$ and $Y$, then we have
$$
d_{GH}(X,Y)=\frac12\inf\bigl\{\dis R:R\in\cR_c(X,Y)\bigr\}.
$$

For any metric space $X$ and a real number $\l>0$, we denote by $\l X$ the metric space which differs from $X$ by multiplication of all its distances by $\l$. For $\l=0$, we define $\l X$ as a single point space.

\begin{thm}[\cite{BBI}]\label{thm:GH_simple}
Let $X$ and $Y$ be metric spaces. Then
\begin{enumerate}
\item\label{thm:GH_simple:1} if $X$ is a single-point metric space, then $d_{GH}(X,Y)=\frac12\diam Y$\rom;
\item\label{thm:GH_simple:2} if $\diam X<\infty$, then
$$
d_{GH}(X,Y)\ge\frac12|\diam X-\diam Y|;
$$
\item\label{thm:GH_simple:3} $d_{GH}(X,Y)\le\frac12\max\{\diam X,\diam Y\}$, in particular,  $d_{GH}(X,Y)<\infty$ for bounded $X$ and $Y$\rom;
\item\label{thm:GH_simple:4} for any metric spaces $X$, $Y$ and any $\l>0$, we have $d_{GH}(\l X,\l Y)=\l d_{GH}(X,Y)$.
\end{enumerate}
\end{thm}

\section{Homogeneous and round metric spaces}
\markright{\thesection.~Homogeneous and round metric spaces}
Fix a real $0<b\le\diam X$ and an integer $n\ge2$. A metric space $X$ is called \emph{$(b,n)$-homogeneous\/} if for any point $x\in X$ there exists a $b$-separated $n$-point subset $S\ss X$ such that $x\in S$.

\begin{thm}\label{thm:homoANDnot}
Let $X$ be a $(b,n)$-homogeneous metric space. Suppose that $Y$ is not $(a,n)$-homogeneous for some $0<a<b$. Then $2d_{GH}(X,Y)\ge b-a$.
\end{thm}

\begin{proof}
Put $c=b-a$ and suppose to the contrary that $2d_{GH}(X,Y)<c$, then there exists $R\in\cR(X,Y)$ such that $\dis R<c$, therefore, for any $(x_1,y_1),\,(x_2,y_2)\in R$ with $|x_1x_2|\ge b$, it holds $|y_1y_2|>|x_1x_2|-c\ge b-c=a$. Since $Y$ is not $a$-homogeneous, there exists $y\in Y$ such that there is no $(a,n)$-separated $n$-point subset $T\ss Y$ with $y\in T$. Since $R$ is a correspondence, there exists $x\in X$ such that $(x,y)\in R$. Since $X$ is $(b,n)$-homogeneous, there exists $b$-separated $S=\{x_1,\ldots,x_n\}\ss X$ such that $x_1=x$. Let $T=\{y_1,\ldots,y_n\}$ with $y_1=y$ and $(x_i,y_i)\in R$ for any $i$. Then $T$ is an $a$-separated $n$-point subset $Y$ containing $y$, a contradiction.
\end{proof}

A metric space $X$ is called \emph{round\/} if it is $(b,2)$-homogeneous for any $0<b<\diam X$.

\begin{examp}
Let $X\ss\R^n$ be a sphere of radius $r>0$ w.r.t\. some norm. Then $X$ is round, but not $(a,2)$-homogeneous for each $a>2r$.
\end{examp}

\begin{cor}\label{cor:Gh-dist-round-nonhomo}
Let $X$ be a round metric space, $0<a<\diam X$, and suppose that a metric space $Y$ is not $(a,2)$-homogeneous. Then $2d_{GH}(X,Y)\ge\diam X-a$.
\end{cor}

\begin{proof}
Theorem~\ref{thm:homoANDnot} implies that $2d_{GH}(X,Y)\ge b-a$ for each $b$, $0<a<b<\diam X$, and the result follows from the arbitrariness of $b$.
\end{proof}

\begin{examp}
Let $X\ss\R^m$ and $Y\ss\R^n$ be spheres of radii $b$ and $a$ w.r.t\. some norms, then
$$
2d_{GH}(X,Y)\ge2b-2a.
$$
\end{examp}

\section{Nonlinearity degree of a metric space}
\markright{\thesection.~Nonlinearity degree of a metric space}

In this section, we introduce the notion of nonlinearity degree and investigate several its properties. To start with, we denote by $\Lip_a(X)$ the set of all real-valued $a$-Lipschitz functions defined on a metric space $X$.

\begin{dfn}
\emph{The nonlinearity degree of a metric space $X$} is defined as follows:
$$
c(X):=\inf_{f\in\Lip_1(X)}\sup\bigl\{|xy|-|f(x)-f(y)|:x,y\in X\bigr\}.
$$
\end{dfn}

\begin{rk}\label{rk:cX}
Clearly, $c(X)\ge0$, and for any $X\ss\R$ we have $c(X)=0$ (indeed, to achieve the value $0$, we can take the inclusion mapping as $f$). Also, since $c(X)\le\diam X$, then for each bounded space $X$, the value $c(X)$ is finite.
\end{rk}

If $f\:X\to\R$ is $a$-Lipschitz, then for any $b\in\R$ the function $f+b$ is $a$-Lipschitz as well, and for any $x,y\in X$ the value $|xy|-|f(x)-f(y)|$ remains the same. In addition, for a bounded metric space, each $a$-Lipschitz function is bounded, thus, to calculate the value $c(X)$ in this case, we can restrict ourselves by those $f\in\Lip_1(X)$ that satisfy $\inf f(X)=0$. The set of all such functions we denote by $\Lip_1^0(X)$. Hence, we got the following

\begin{thm}\label{thm:LipschitzWithZero}
For any bounded metric space $X$, we have
$$
c(X):=\inf_{f\in\Lip_1^0(X)}\sup\bigl\{|xy|-|f(x)-f(y)|:x,y\in X\bigr\}.
$$
\end{thm}

The nonlinearity degree can be used to estimate the minimal Gromov--Hausdorff distance between metric space $X$ and non-empty subsets of $\R$.

\begin{thm}\label{thm:GHidtsToRsubs}
For any metric space $X$, we have
$$
\bigl|X\,\cP_0(\R)\bigr|=\inf_{Z\in\cP_0(\R)}d_{GH}(X,Z)\le c(X)/2.
$$
If $X$ is bounded, then there exists $Z\in\cH(\R)$ such that $d_{GH}(X,Z)\le c(X)/2$.
\end{thm}

\begin{proof}
If $c(X)=\infty$ then the result holds.

Now, suppose that $c(X)<\infty$. Take $f_k\in\Lip_1(X)$ such that
$$
\sup\Bigl\{|xy|-\bigl|f_k(x)-f_k(y)\bigr|:x,y\in X\Bigr\}\le c(X)+\frac1k,
$$
and put $X_k=f_k(X)$. Let $R_k\in\cR(X,X_k)$ be the graph of $f_k$, namely,
$$
R_k=\Bigl\{\bigl(x,f_k(x)\bigr): x\in X\Bigr\}.
$$
Since $f_k$ is $1$-Lipschitz, then for any $\bigl(x,f_k(x)\bigr),\,\bigl(y,f_k(y)\bigr)\in R_k$ we have $|f_k(x)-f_k(y)|\le|xy|$, therefore,
$$
\dis R_k=\sup_{x,y\in X}\Bigl\{|xy|-\bigl|f_k(x)-f_k(y)\bigr|\Bigr\}\le c(X)+\frac1k.
$$
Due to the arbitrariness of $k\in\N$, we get
$$
\inf_{Z\in\cP_0(\R)}d_{GH}(X,Z)\le\inf_kd_{GH}(X,X_k)\le c(X)/2.
$$

Now we prove the existence of $Z$ for bounded $X$. By Theorem~\ref{thm:LipschitzWithZero}, we can take $f_k$ from $\Lip_1^0(X)$ instead of from $\Lip_1(X)$. Let $\bX_k$ be the closure of $X_k=f_k(X)$. Since $d_{GH}(X_k,\bX_k)=0$, then we have $\inf_kd_{GH}(X,\bX_k)\le c(X)/2$ by the triangle inequality. The boundedness of $X$ implies $d=\diam X<\infty$, and by Theorem~\ref{thm:GH_simple}, we have
$$
\diam\bX_k=\diam X_k\le\diam X+2d_{GH}(X,\bX_k)\le d+c(X)<\infty.
$$
Since $X_k\ge 0$ and $0\in\bX_k$ for each $k\in\N$, then $X_k\ss\bigl[0,d+c(X)\bigr]=:I$ for all $k\in\N$, and, by compactness of $\cH(I)$ due to Theorem~\ref{thm:HausdorfProps}, there exists a convergent subsequence $\bX_{k_i}\in\cH(I)$. Put $Z=\lim_{i\to\infty}\bX_{k_i}$. Since $d_{GH}(\bX_{k_i},Z)\le d_H(\bX_{k_i},Z)$, we obtain $\bX_{k_i}\toGH Z$, therefore $d_{GH}(X,Z)\le c(X)/2$.
\end{proof}

\begin{examp}
Now we show that the inequality in Theorem~\ref{thm:GHidtsToRsubs} cannot be changed to equality, even in the case of compact $X$. Let $X$ be the standard unit circle in the Euclidean plane, endowed with the intrinsic metric. We shall find a compact
$Z\ss\R$ such that $d_{GH}(X,Z)<c(X)/2$. Let us take as $Z$ the segment $[0,\frac23\pi]$. Below, in Lemma~\ref{lem:CircleSegmentLowerBound}, we shall prove that $d_{GH}(X,Z)=\frac13\pi$. On the other hand, given any $f\in\Lip_1(S^1,\R)$, the function $g(x)=f(x)-f(-x)$ has at least one zero point $x_0$. Then
$$
\sup\Bigl\{|xy|-\bigl|f(x)-f(y)\bigr|:x,y\in X\Bigr\}\ge|x_0(-x_0)|-\bigl|f(x_0)-f(-x_0)\bigr|=\pi.
$$
Since the choice of $f$ is arbitrary, we know that $c(X)\ge\pi>2d_{GH}(X,Z)$.
\end{examp}

The next result characterizes all compact metric spaces with $c(X)=0$.

\begin{cor}\label{cor:linearzero}
Each compact metric space $X$ with $c(X)=0$ is isometric to a compact subset of $\R$.
\end{cor}

\begin{proof}
By Theorem~\ref{thm:GHidtsToRsubs}, there exists a compact subset $Z\ss\R$ with $d_{GH}(X,Z)=0$, thus $X$ and $Z$ are isometric.
\end{proof}

\begin{prop}\label{prop:Min for c exists}
For a compact metric space $X$, there exists a function $f\in\Lip_1(X)$ such that
$$
\sup\bigl\{|xy|-|f(x)-f(y)|:x,y\in X\bigr\}=c(X).
$$
\end{prop}

\begin{proof}
By Theorem~\ref{thm:LipschitzWithZero}, we can choose a sequence $f_n\in\Lip_1^0(X)$ such that for each $n\in\N$ it holds
$$
\sup\bigl\{|xy|-|f_n(x)-f_n(y)|:x,y\in X\bigr\}<c(X)+\frac1n.
$$

The next lemma is a direct consequence of Arzel\`a-Ascoli theorem~\cite{BBI}.

\begin{lem}\label{lem:ArzelaAscoli}
Each uniformly bounded sequence of $L$-Lipschitz functions on a compact metric space contains a subsequence uniformly convergent to an $L$-Lipschitz function.
\end{lem}

Since any collection of functions from $\Lip_1(X)$ is uniformly bounded, we can apply Lemma~\ref{lem:ArzelaAscoli}, thus, w.l.o.g\. we can suppose that the sequence $f_n$ itself uniformly converges to some $f\in\Lip_1(X)$. Take arbitrary $\e>0$ and choose $N\in\N$ such that for any $n\ge N$ and any $x\in X$ we have
$$
\bigl|f(x)-f_n(x)\bigr|<\e/3\ \ \text{and}\ \ 1/n<\e/3,
$$
hence for any $x,y\in X$ and any $n\ge N$ it holds
$$
\sup\bigl\{|xy|-|f(x)-f(y)|:x,y\in X\bigr\}\le\sup\bigl\{|xy|-|f_n(x)-f_n(y)|:x,y\in X\bigr\}+2\e/3<c(X)+\e.
$$
Due to the arbitrariness of $\e$, we have $\sup\bigl\{|xy|-|f(x)-f(y)|:x,y\in X\bigr\}\le c(X)$, that concludes the proof.
\end{proof}

\begin{rk}
Proposition~\ref{prop:Min for c exists} gives another proof of Theorem $\ref{thm:GHidtsToRsubs}$ when $X$ is compact.
\end{rk}

\begin{thm}\label{thm:circle-cvalued}
Suppose we have a connected compact metric space $X$ equipped with a homeomorphism $\a$ of order $2$ such that $\bigl|\a(x)\,x\bigr|=\diam X$ for each $x\in X$, and a compact metric space $Y$ such that $c(Y)<\diam X$. Then we have $2d_{GH}(X,Y)\ge m$, where $m=\frac23\bigl(\diam X-c(Y)\bigr)$.
\end{thm}

\begin{proof}
Suppose to the contrary that $2d_{GH}(X,Y)<m$, thus there is a closed correspondence $R\in\cR_c(X,Y)$ such that $\dis R<m$.

By Proposition $\ref{prop:Min for c exists}$, there exists $v\in \Lip_1(Y)$ such that
$$
\sup\bigl\{|y_1y_2|-|v(y_1)-v(y_2)|:y_1,y_2\in Y\bigr\}=c(Y).
$$
In the following text, we will abbreviate $c(Y)$ as $c$ and $\a(x)$ as $\a_x$.

Since $X\x Y$ is compact and $R$ is its closed subset, then $R$ is compact, together with $R(x)=R\cap\bigl(\{x\}\x Y\bigr)$ and $R(\a_x)=R\cap\bigl(\{\a_x\}\x Y\bigr)$. Since the function $v$ is continuous, then the restrictions of $v$ to $R(x)$ and $R(\a_x)$ are bounded and attain their maximal and minimal values, thus the following function $f\:X\to\R$ is correctly defined:
$$
f(x)=\max v\bigl(R(x)\bigr)-\min v\bigl(R(\a_x)\bigr).
$$

\begin{lem}\label{lem:approximation}
For any sequence $x_i\in X$ such that $x_i\to\tx\in X$ it holds
$$
\limsup_{i\to+\infty}f(x_i)\le f(\tx).
$$
\end{lem}

\begin{proof}
Let us put $R^v:=\bigl\{\bigl(x,v(y)\bigr):(x,y)\in R\bigr\}\ss X\x\R$, then we can rewrite $f$ as follows:
$$
f(x)=\max R^v(x)-\min R^v(\a_x).
$$
Since $R$ is compact, then $R^v$ is compact as the image of $R$ under the continuous mapping $\id\x v$. Thus $R^v$ is closed and bounded.

We put $z_i=\max R^v(x_i)$, $w_i=\min R^v(\a_{x_i})$,  then $f(x_i)=z_i-w_i$, both the sequences $z_i$, $w_i$ are bounded, and both the $z:=\limsup_{i\to+\infty}z_i$, $w:=\liminf_{i\to+\infty}w_i$ are finite. Now we choose subsequences $z_{k_i}$ and $w_{l_i}$ such that
$$
z=\lim_{i\to+\infty}z_{k_i}\ \ \text{and}\ \ w=\lim_{i\to+\infty}w_{l_i}.
$$
Since $R^v$ is closed, $(x_{k_i},z_{k_i}),(\a_{x_{l_i}},w_{l_i})\in R^v$, and $z$, $w$ are finite, then
$$
(\tx,z)=\lim_{i\to +\infty}(x_{k_i},z_{k_i})\in R^v\ \ \text{and}\ \
(\a_\tx,w)=\lim_{i\to+\infty}(\a_{x_{l_i}},w_{l_i})\in R^v,
$$
therefore, $z\in R^v(\tx)$ and $w\in R^v(\a_\tx)$. Thus, we get
$$
\max R^v(\tx)\ge z=\limsup_{i\to+\infty}z_i\ \ \text{and}\ \ \min R^v(\a_\tx)\le w=\liminf_{i\to+\infty}w_i.
$$
In account,
\begin{multline*}
f(\tx)\ge z-w=\limsup_{i\to+\infty}z_i-\liminf_{i\to+\infty}w_i=\\
=\limsup_{i\to+\infty}z_i+\limsup_{i\to+\infty}(-w_i)\ge \limsup_{i\to+\infty}(z_i-w_i)=\limsup_{i\to+\infty}f(x_i),
\end{multline*}
and the proof is completed.
\end{proof}

Since $\dis R<m$ and $|x\,\a_x|=\diam X$, then for any $y_1\in R(x)$ and $y_2\in R(\a_x)$ we have $|y_1y_2|>\diam X-m$, thus
$$
\bigl|v(y_1)-v(y_2)\bigr|\ge|y_1y_2|-c>\diam X-m-c=:d.
$$
So, either $v(y_1)-v(y_2)>d$ or $v(y_2)-v(y_1)>d$. Since the values $\max v\bigl(R(x)\bigr)$ and $\min v\bigl(R(\a_x)\bigr)$ are attained, then for each $x\in X$ it holds
either $f(x)>d$ or $-f(x)>d$, and the latter is equivalent to $f(x)<-d$.

We put $A:=\bigl\{x\in X:f(x)>d\bigr\}$ and $B:=\bigl\{x\in X:f(x)<-d\bigr\}$. Since $c\le\diam X$ and $m=\frac23(\diam X-c)$, we get
$$
d=\diam X-m-c=\frac13(\diam X-c)\ge0\ge-\frac13(\diam X-c)=c+m-\diam X=-d,
$$
hence $A$ is the complement of $B$. Moreover, since $f(x)\neq\pm d$, we have
$$
A=\bigl\{x\in X:f(x)\ge d\bigr\}\ \ \text{and}\ \ B=\bigl\{x\in X:f(x)\le-d\bigr\}.
$$

Now we show that $A$ is closed and, thus, $B$ is open. Indeed, consider an arbitrary sequence $x_i\in A$ converging to some $x\in X$. Then, by Lemma~\ref{lem:approximation}, we have $f(x)\ge\limsup\limits_{x_i\to x}f(x_i)\ge d$, thus $x\in A$ and, therefore, $A$ is closed.

\begin{lem}\label{lem:alphaAssB}
We have $\a(B)\ss A$.
\end{lem}

\begin{proof}
Take arbitrary $x\in B$, then, by definition, $\max v\bigl(R(x)\bigr)-\min v\bigl(R(\a_x)\bigr)\le-d$, thus
$$
-f(x)=\min v\bigl(R(x)\bigr)-\max v\bigl(R(\a_x)\bigr)\le\max v\bigl(R(x)\bigr)-\min v\bigl(R(\a_x)\bigr)\le-d,
$$
so $f(x)\ge d$ and, hence, $\a_x\in A$.
\end{proof}

\begin{lem}\label{lem:alphaBssA}
It holds that $\a(A)\subset B$.
\end{lem}

\begin{proof}
If not, there exists $x\in A$ such that $\a_x\in A$. Without loss of generality, suppose that $\max v\bigl(R(\a_x)\bigr)\ge\max v\bigl(R(x)\bigr)$. Since $x\in A$, we have, by definition, $\max v\bigl(R(x)\bigr)\ge d+\min v\bigl(R(\a_x)\bigr)$.

Recall that for any $y\in R(x)$ and $y'\in R(\a_x)$ it holds $\bigl|v(y)-v(y')\bigr|>d$, thus we get
$$
\max v\bigl(R(\a_x)\bigr)\ge d+\max v\bigl(R(x)\bigr)\ge2d+\min v\bigl(R(\a_x)\bigr).
$$
As we mentioned above, $R(\a_x)$ is compact and $v$ is continuous, hence there exist $y_1,y_2\in R(\a_x)$ such that $v(y_1)=\max v\bigl(R(\a_x)\bigr)$ and $v(y_2)=\min v\bigl(R(\a_x)\bigr)$. Since $v$ is $1$-Lipschitz, we have $|y_1y_2|\ge v(y_1)-v(y_2)$, therefore,
\begin{multline*}
m>\dis R\ge\diam R(\a_x)\ge|y_1y_2|\ge v(y_1)-v(y_2)=\\
=\max v\bigl(R(\a_x)\bigr)-\min v\bigl(R(\a_x)\bigr)\ge2d=2\bigl(\diam X-m-c\bigr).
\end{multline*}
Thus $m>\frac23(\diam X-c)$ which achieves contradiction and, hence, we have $\a(A)\ss B$.
\end{proof}

Since $X=A\sqcup B$ and $\a\:X\to X$ is a homeomorphism, Lemmas~\ref{lem:alphaAssB} and~\ref{lem:alphaBssA} imply that the restriction of $\a$ to $A$ is a homeomorphism onto $B$, hence both $A$ and $B$ are proper closed-open, therefore, $X$ is not connected, a contradiction which completes the proof of the theorem.
\end{proof}

\section{Geometric calculation of distortion}\label{sec:geomCalc}
\markright{\thesection.~Geometric calculation of distortion}
Consider the special case we are interesting in, namely, let $X=I_\l\ss\R$ be a segment of the length $\l$, and $Y=S^1=\bigl\{(x,y)\in\R^2:x^2+y^2=1\bigr\}$ be the standard circle endowed with the intrinsic metric. Since we are interesting in the Gromov--Hausdorff distance between $I_\l$ and $S^1$, it does not matter where we place $I_\l$ on the line $\R$. We will use two possible placements: $[0,\l]$ and $[-\l/2,\,\l/2]$. In this section we will use the latter one.

Take a correspondence $R\in\cR(I_\l,S^1)$. Our goal is to calculate geometrically the distortion of $R$. To do that, we parameterize the circle $S^1$ by the standard angular coordinate $-\pi\le\v\le\pi$, thus the interior distance between points $\v_1,\v_2\in S^1$ equals $\min\bigl\{|\v_1-\v_2|,2\pi-|\v_1-\v_2|\bigr\}$. Let $-\l/2\le t\le\l/2$ be the standard coordinate on the segment $I_\l$. Now we represent the correspondence $R$ as a subset of rectangle $Q:=[-\l/2,\,\l/2]\x[-\pi,\pi]\ss\R^2$.

To calculate the distortion, we need to maximize the function
$$
f\bigl((t_0,\v_0),(t,\v)\bigr)=
\begin{cases}
\bigl|\,|t-t_0|-|\v-\v_0|\,\bigr|&\text{if $|\v-\v_0|\le\pi$},\\
\bigl|\,|t-t_0|-2\pi+|\v-\v_0|\,\bigr|&\text{if $|\v-\v_0|\ge\pi$}\\
\end{cases}
$$
over all $(t_0,\v_0),\,(t,\v)\in R$. In other words, we need to find the least possible $a$ such that $f\bigl((t_0,\v_0),(t,\v)\bigr)\le a$ for all $(t_0,\v_0),\,(t,\v)\in R$.

Take arbitrary $a>0$, and denote by $D_a(t_0,\v_0)$ the subset of the plane $\R^2$ consisting of all $(t,\v)\in\R^2$ such that $f\bigl((t_0,\v_0),(t,\v)\bigr)\le a$. Clearly that $D_a(t_0,\v_0)$ can be obtained from $D_a:=D_a(0,0)$ by shifting by $(t_0,\v_0)$, i.e., $D_a(t_0,\v_0)=(t_0,\v_0)+D_a$. Evident observation shows that $D_a$ looks like the blue domain in Figure~\ref{fig:distort}, left-hand side.

\ig{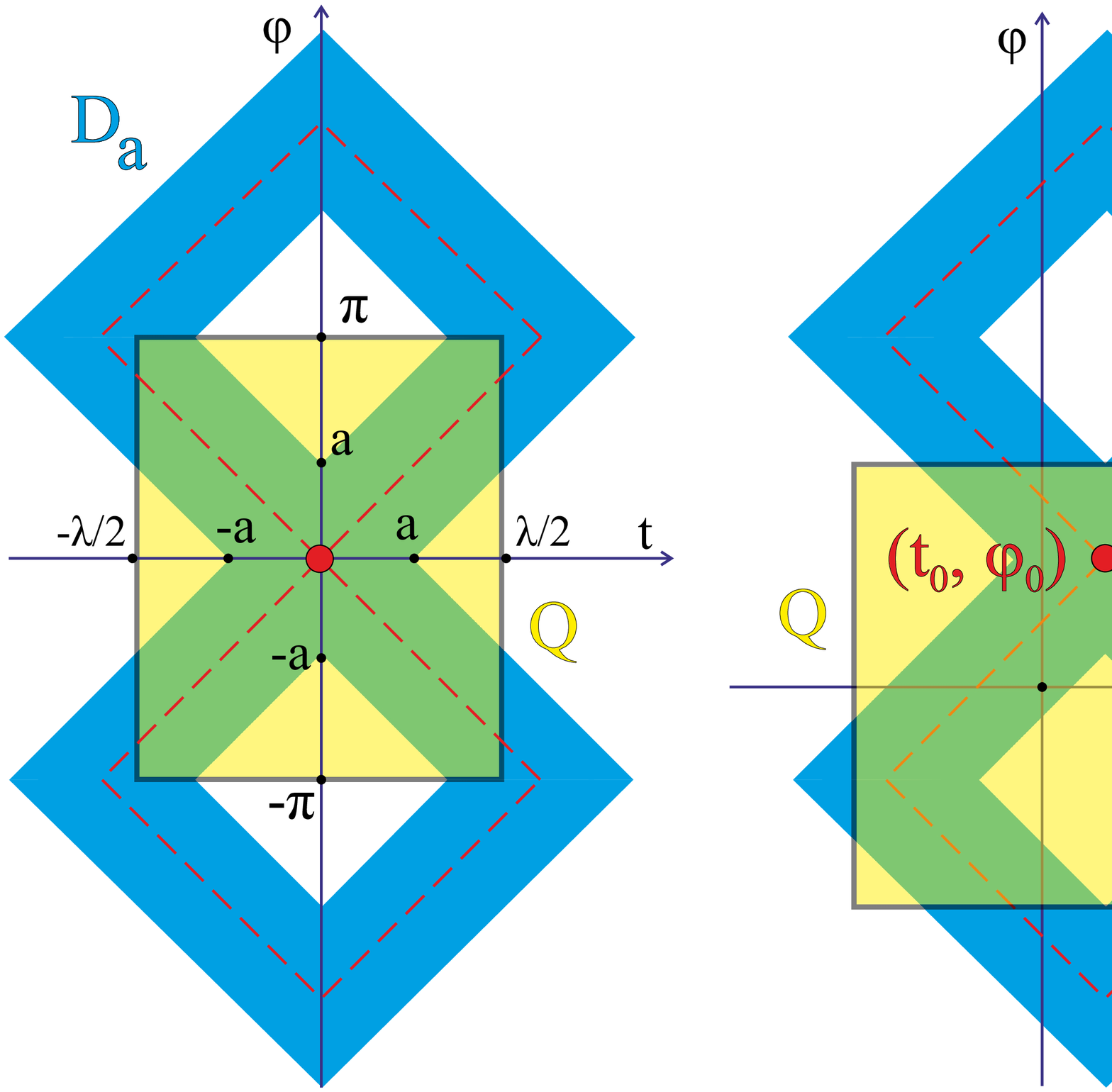}{0.4}{fig:distort}{Geometric calculation of distortion.}

It is bounded by the segments parallel to bisectors of quadrants, and its sizes are given by the labels in the figure. Thus, we get the following result.

\begin{thm}\label{thm:geomCalc}
Under the notations introduced above, $\dis R\le a$ if and only if for any $(t_0,\v_0)\in R$ it holds $R\ss D_a(t_0,\v_0)$. In particular, the distortion of $R\ss Q$ equals the least possible $a$ satisfying the previous condition.
\end{thm}
\begin{rk}
If a graph $P$ satisfies the above property with respect to $a$, then any subgraph of it also satisfies the above property.
\end{rk}

Let us note that the domain $Q\cap D_a(t_0,\v_0)$ is depicted on the right-hand side of Figure~\ref{fig:distort} in green, as the intersection of blue and yellow domains.

\section{Calculating GH-distance between circle and segment}
\markright{\thesection.~Calculating GH-distance between circle and segment}

Now we apply the technique described above to calculate the Gromov--Hausdorff distance between the standard circle  $S^1=\bigl\{(x,y)\in\R^2:x^2+y^2=1\bigr\}$ with intrinsic metric and the segment $I_\l=[0,\l]\ss\R$. We start from the following

\begin{lem}\label{lem:CircleSegmentLowerBound}
For each $0\le\l<\pi$ we have
$$
d_{GH}(I_\l,S^1)\ge\frac{\pi}2-\frac\l4.
$$
\end{lem}

\begin{proof}
Notice that the metric space $S^1$ is round, $\diam S^1=\pi$, and for any $\l'>\l$ the segment $I_\l$ is not $(\l'/2,2)$-homogeneous. Thus, for any $\pi>\l'>\l$ we can use Corollary~\ref{cor:Gh-dist-round-nonhomo} and get
$$
d_{GH}(I_\l,S^1)\ge\frac{\pi}2-\frac{\l'}4.
$$
The result follows from arbitrariness of $\l'$.
\end{proof}

\begin{prop}\label{prop:0topi3}
For each $0\le\l\le2\pi/3$ we have
$$
d_{GH}(I_\l,S^1)=\frac\pi2-\frac\l4.
$$
\end{prop}

\begin{proof}
By Lemma~\ref{lem:CircleSegmentLowerBound}, we get the necessary lower bound for the $d_{GH}(I_\l,S^1)$. It remains to get the same upper bound. To do that, we will construct a special correspondence $R_\l\in\cR(I_\l,S^1)$ with $\dis R_\l=\pi-\l/2$, which completes the proof.

For $\l=0$ we put $R_\l=I_\l\x S^1$. Since the space $I_\l$ consists of a single point, we get
$$
\dis R_\l=\diam S^1=\pi=\pi-\l/2.
$$

Now, suppose that $\l>0$. Identify $\R^2$ with $\C$, and let $R_\l$ be the graph of the mapping $\g_\l\:I_\l\to S^1$ defined as $\g\:t\mapsto e^{2\pi i\,t/\l}$. Notice that the image of the mapping $t\mapsto2\pi t/\l$, $t\in I_\l$, equals $\{e^{it}:t\in[0,2\pi]\}$, thus $R_\l$ is a correspondence, and for any $t,t'\in I_\l$, we have
$$
\bigl|\g(t)\g(t')\bigr|=\min\Bigl\{\frac{2\pi}\l|t-t'|,\,2\pi-\frac{2\pi}\l|t-t'|\Bigr\}.
$$
Since the value $\bigl|\g(t)\g(t')\bigr|$ depends only on $s:=|t-t'|$, we put $f(s)=\bigl|\g(t)\g(t')\bigr|-|t-t'|$, then
$$
\dis R_\l=\max_{0\le s\le\l}\bigl|f(s)\bigr|.
$$
Since the function
$$
f(s)=
\begin{cases}
\frac{2\pi}\l s-s&\text{for $0\le\frac{2\pi}\l s\le\pi$},\\
2\pi-\frac{2\pi}\l s-s&\text{for $\pi\le\frac{2\pi}\l s\le2\pi$},
\end{cases}
$$
is linear on $[0,\l/2]$ and on $[\l/2,\l]$, we get
$$
\max_{0\le s\le\l}\bigl|f(s)\bigr|=\max\Bigl\{\bigl|f(0)\bigr|,\,\bigl|f(\l/2)\bigr|,\,\bigl|f(\l)\bigr|\Bigr\}=
\max\bigl\{0,\,|\pi-\l/2|,\,\l\bigr\}.
$$
Since $0\le\l\le2\pi/3$, then $|\pi-\l/2|=\pi-\l/2\ge\l>0$, therefore $\dis R_\l=\pi-\l/2$.
\end{proof}

\begin{prop}\label{prop:23pito76pi}
If $\frac23\pi\le\l\le\frac76\pi$, then we have $d_{GH}(I_\l,S^1)=\frac\pi3$.
\end{prop}

\begin{proof}
Applying Lemma~\ref{lem:leq-of-circle-and-cvalued}, it suffices to construct a correspondence $R_\l\in\cR(I_\l,S^1)$ with $\dis R_\l=2\pi/3$, which completes the proof.

Again, we identify $\R^2$ with $\C$, and let $R_\l$ be the graph of the mapping $\g_\l\:I_\l\to S^1$ defined as $\g\:t\mapsto e^{3\pi i\,t}$. Notice that the image of the mapping $t\mapsto3\pi t$, $t\in I_\l$, contains $\{e^{it}:t\in[0,2\pi]\}$, thus $R_\l$ is a correspondence, and for any $t,t'\in I_\l$, we have
$$
\bigl|\g(t)\g(t')\bigr|=\min\Bigl\{3|t-t'|,\,2\pi-3|t-t'|,\,3|t-t'|-2\pi,\,4\pi-3|t-t'|\Bigr\}.
$$
Since the value $\bigl|\g(t)\g(t')\bigr|$ depends only on $s:=|t-t'|$, we put $f(s)=\bigl|\g(t)\g(t')\bigr|-|t-t'|$, then
$$
\dis R_\l=\max_{0\le s\le\l}\bigl|f(s)\bigr|.
$$
Since the function
$$
f(s)=
\begin{cases}
2s&\text{for $0\le s\le\frac\pi3$},\\
2\pi-4s&\text{for $\frac\pi3\le s\le\frac23\pi$},\\
2s-2\pi&\text{for $\frac23\pi\le s\le\pi$},\\
4\pi-4s&\text{for $\pi\le s\le\frac76\pi$},\\
\end{cases}
$$
is linear on $[0,\pi/3]$, $[\pi/3,2\pi/3]$, $[2\pi/3,\pi]$, $[\pi,7\pi/6]$, we get
$$
\max_{0\le s\le7\pi/6}\bigl|f(s)\bigr|=\max\Bigl\{\bigl|f(0)\bigr|,\,\bigl|f(\pi/3)\bigr|,\,\bigl|f(2\pi/3)\bigr|,\, \bigl|f(\pi)\bigr|,\,\bigl|f(7\pi/6)\bigr|\Bigr\}=
\max\{0,\,2\pi/3\}=2\pi/3,
$$
therefore, $\dis R_\l=2\pi/3$.
\end{proof}

\begin{prop}\label{prop:morethan5pi3}
If $5\pi/3\le\l\le2\pi$, then we have $d_{GH}(I_\l,S^1)=\frac{\l-\pi}2$.
\end{prop}

\begin{proof}
By Theorem~\ref{thm:GH_simple}, we have
$$
d_{GH}(I_\l,S^1)\ge\frac{\diam I_\l-\diam S^1}2=\frac{\l-\pi}2.
$$
Thus, it suffices to construct a correspondence $R\in\cR(I_\l,S^1)$ with $\dis R\le\l-\pi=:a$. We will use notations from Section~\ref{sec:geomCalc} and apply Theorem~\ref{thm:geomCalc}.

Consider the correspondence depicted in Figure~\ref{fig:corresp5Pi3}.

\ig{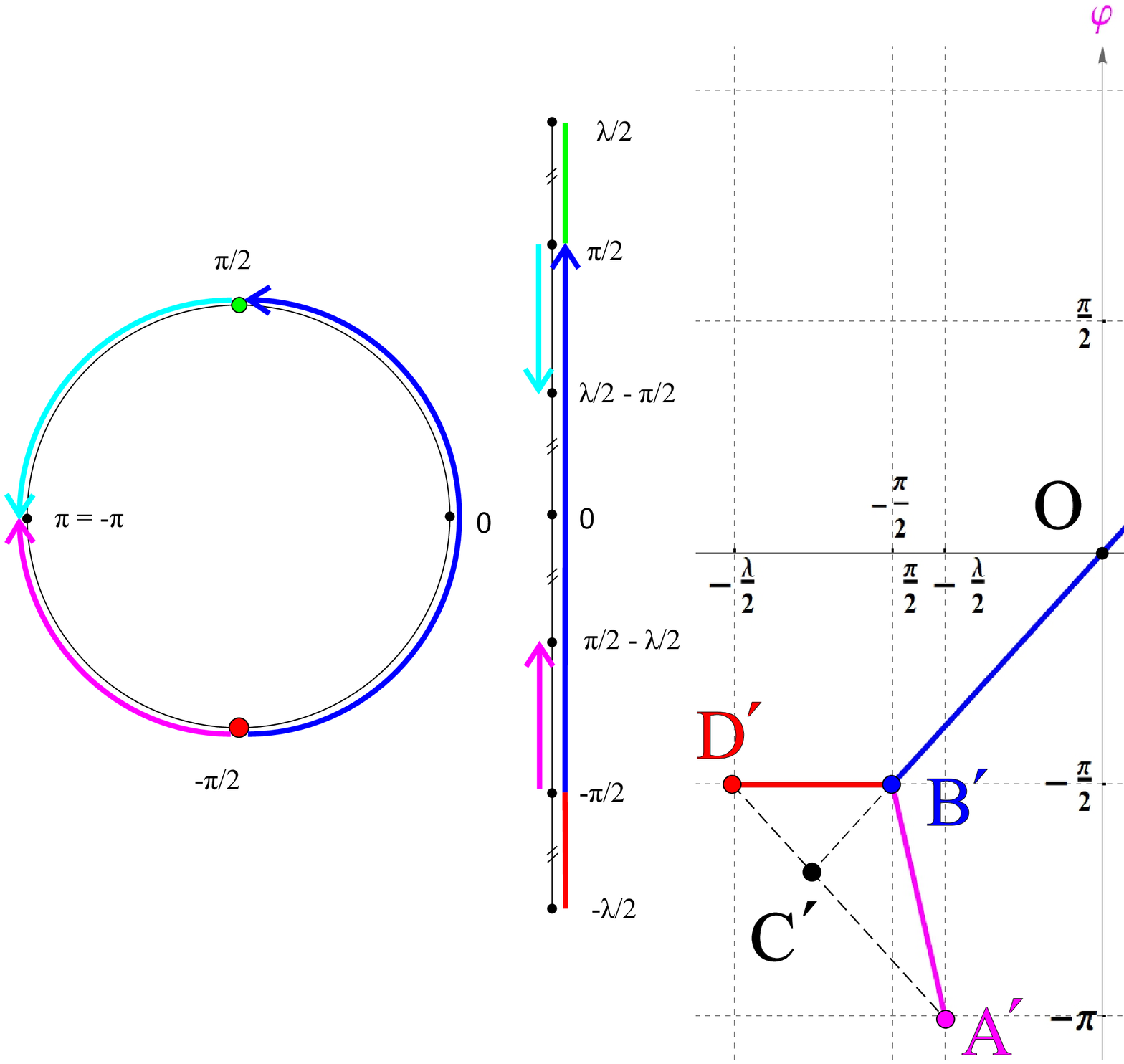}{0.5}{fig:corresp5Pi3}{A correspondence $R$ to prove Proposition~\ref{prop:morethan5pi3}.}

We have to show that for any $(t_0,\v_0)\in R\ss[-\l/2,\l/2]\x[-\pi,\pi]$ the domain $D_a(t_0,\v_0)$ covers $R$. In Figure~\ref{fig:distort5Pi3} we show the corresponding domain $D_a$ endowed with its special sizes; also we write down the corresponding sizes of segments forming $R$.

\ig{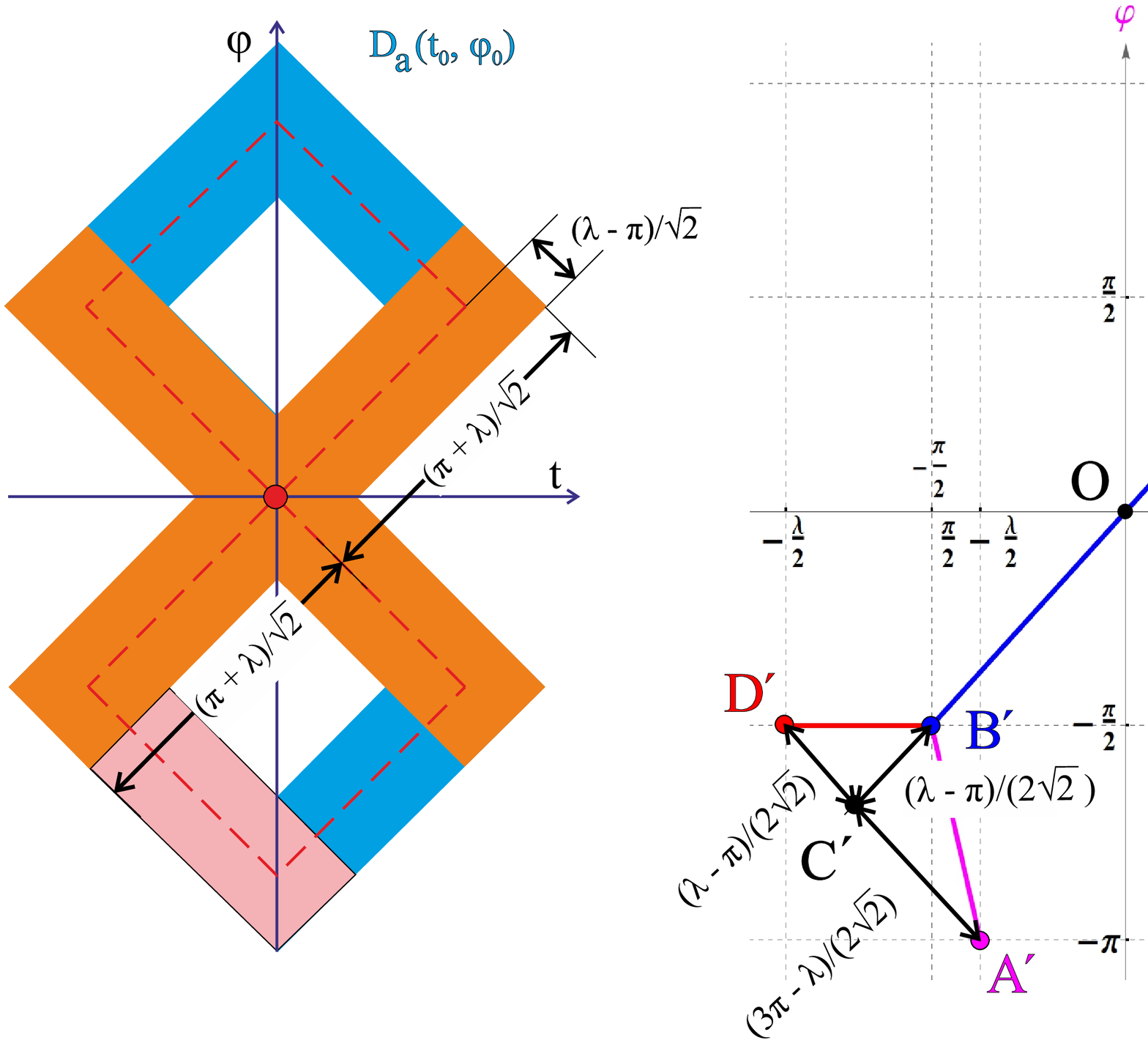}{0.4}{fig:distort5Pi3}{A correspondence $R$ to prove Proposition~\ref{prop:morethan5pi3}.}

Due to symmetry reasons, it suffices to verify the cases of $(t_0,\v_0)$ lying on blue, green or cyan segments of $R$. We show that in these cases $R$ is covered by orange cross and rose rectangle from the left-hand side of Figure~\ref{fig:distort5Pi3}.

To simplify verification of the formulas below, let us explicitly write down the coordinates of the points $A$, $B$, $C$, and $D$ from Figure~\ref{fig:corresp5Pi3} and~\ref{fig:distort5Pi3} (points with the same names supplied by primes are symmetric w.r.t. the origin):
$$
A=\Bigl(\frac\l2-\frac\pi2,\pi\Bigr),\ \ B=\Bigl(\frac\pi2,\frac\pi2\Bigr),\ \ C=\Bigl(\frac\l4+\frac\pi4,\frac\l4+\frac\pi4\Bigr),\ \ D=\Bigl(\frac\l2,\frac\pi2\Bigr).
$$
It is easy to calculate that
$$
|AC|=\frac{3\pi-\l}{2\sqrt2}\leq\frac{\l-\pi}{\sqrt2},\ \
|CD|=\frac{\l-\pi}{2\sqrt2}<\frac{\l-\pi}{\sqrt2},\ \
|B'C|=|BC'|=\frac{\l+3\pi}{2\sqrt2}<\frac{\l+\pi}{\sqrt2}.
$$
Since the length and width of each of the two rectangles forming the orange cross are $\sqrt2(\l+\pi)$ and $\sqrt2(\l-\pi)$, respectively, then, once the point $(t_0,\phi_0)$ lies on the blue segment of the relation $R$, the orange cross covers $R$.

Now, consider the case when $(t_0,\phi_0)$ lies on the green segment. Still, the whole correspondence $R$ is covered by the orange cross. To see that, it suffices to notice the following evident formulas:
\begin{gather*}
|CC'|=\frac{\l+\pi}{\sqrt2},\ \
|CD|+|C'D'|=\frac{\l-\pi}{\sqrt2},\ \
|A'C'|-|CD|=\frac{2\pi-\l}{\sqrt2}<\frac{\l-\pi}{\sqrt2},\\
|BC|=\frac{\l-\pi}{2\sqrt2}<\frac{\l-\pi}{\sqrt2},\ \
|AD|=\frac\pi{\sqrt2}<\frac{\l+\pi}{\sqrt2}.
\end{gather*}

At last, let $(t_0,\phi_0)$ lie on the cyan segment. Then, similarly, the whole correspondence $R$ except the magenta segment is covered by the orange cross. Let us show that the magenta segment is covered by the both orange cross and the rose rectangle. This can be extracted from the following inequalities:
$$
|AC|+|A'C'|=\frac{3\pi-\l}{\sqrt2}<\frac{\l+\pi}{\sqrt2},\ \
|BB'|=\sqrt2\pi<\sqrt2\pi-\frac{\l-\pi}{\sqrt2},\ \
|CC'|= \frac{\l+\pi}{\sqrt2}.
$$
The proof is completed.
\end{proof}

\begin{lem}\label{lem:leq-of-circle-and-cvalued}
For each $\pi\le\l\le\frac{5\pi}3$ we have $d_{GH}(I_\l,S^1)\ge\frac\pi3$.
\end{lem}

\begin{proof}
Since $I_\l$ is a subset of $\R$, we have $c(I_\l)=0$ by Remark~\ref{rk:cX}.

For $x\in S^1$, denote the antipodal point of $x$ by $\a(x)$, then $\a\:S^1\to S^1$ is a homeomorphism of order $2$ and $\bigl|x\a(x)\bigr|=\diam S^1$. Since $S^1$ is connected and compact, and $I_\l$ is compact, and $c(I_\l)=0<\pi=\diam S^1$, we can apply Theorem~\ref{thm:circle-cvalued} to the pair $(S^1,I_\l)$ and, thus, we get
$$
2d_{GH}(I_\l,S^1)\ge\frac{2\pi}3.
$$
\end{proof}

\begin{prop}
If $\pi\le\l\le\frac{5\pi}3$, then we have $d_{GH}(I_\l,S^1)=\frac{\pi}3$.
\end{prop}

\begin{proof}
Lemma~\ref{lem:leq-of-circle-and-cvalued} implies that it suffices to prove the inequality $d_{GH}(I_\l,S^1)\le\frac\pi3$. To do that, we construct a similar correspondence $R'_\l\in\cR(I_\l,S^1)$ as in the proof of Proposition~\ref{prop:morethan5pi3} as follows. Compared with the previous correspondence, the endpoints of the cyan (respectively, magenta) segment are $(\frac\pi3,\pi)$ and $(\frac\pi2,\frac\pi2)$ (respectively, $(-\frac\pi3,-\pi)$ and $(-\frac\pi2,-\frac\pi2)$).

Let us put $P=R'_{\frac{5\pi}{3}}$. Note that $P$ is exactly the same construction for $\lambda=\frac{5\pi}{3}$ as in Proposition~\ref{prop:morethan5pi3}. Thus, $P$ satisfies the condition of Theorem~\ref{thm:geomCalc} with respect to $a=\frac{2\pi}{3}$.

Note that for $\l\in [\pi,\frac{5\pi}3]$, the correspondence $R'_\l$ is a subgraph of $P$ since $R'_\lambda$ only has shorter green and red segments compared with $P$. Since $P$ satisfies Theorem~\ref{thm:geomCalc} for $a=\frac{2\pi}3$, we conclude that $R'_\l$ also satisfies Theorem~\ref{thm:geomCalc} for $a=\frac{2\pi}{3}$. The proof is completed.
\end{proof}

\begin{prop}\label{prop:morethan2pi}
If $\l\ge2\pi$, then we have $d_{GH}(I_\l,S^1)=\frac{\l-\pi}2$.
\end{prop}

\begin{proof}
Again, by Theorem~\ref{thm:GH_simple}, we have
$$
d_{GH}(I_\l,S^1)\ge\frac{\diam I_\l-\diam S^1}2=\frac{\l-\pi}2.
$$
Let $Y=\{e^{it}:0<t<2\}\ss\C$ be the standard unit circle, and
$$
X=\Bigl[-1-\frac{\l-\pi}2,-1\Bigr]\cup\ Y\cup\Bigl[1,1+\frac{\l-\pi}2\Bigr]\ss\C
$$
is equipped with intrinsic metric. Then $S^1$ is isometric to $Y$, and $I_\l$ is isometric to
$$
Z=\Bigl[-1-\frac{\l-\pi}2,-1\Bigr]\cup\ \{e^{it}:\pi\le t\le 2\pi\}\cup\Bigl[1,1+\frac{\l-\pi}2\Bigr]\ss X.
$$
Clearly that $d_H(Y,Z)=\frac{\l-\pi}2$, thus
$$
d_{GH}(I_\l,S^1)\le\frac{\l-\pi}2.
$$
\end{proof}

Gathering together Propositions~\ref{prop:0topi3}--\ref{prop:morethan2pi}, we get

\begin{thm}
Let $S^1=\bigl\{(x,y)\in\R^2:x^2+y^2=1\bigr\}$ be the standard circle with intrinsic metric, and $I_\l=[0,\l]\ss\R$ the segment of the length $\l$. Then
$$
d_{GH}(I_\l,S^1)=
\begin{cases}
\dfrac\pi2-\dfrac\l4& for\ \ 0\le\l\le\dfrac23\pi,\\[10pt]
\dfrac\pi3& for\ \ \dfrac23\pi\le\l\le\dfrac53\pi,\\[10pt]
\dfrac{\l-\pi}2& for\ \ \l\ge\frac53\pi,
\end{cases}
$$
see Figure~$\ref{fig:figure}$.
\end{thm}

\ig{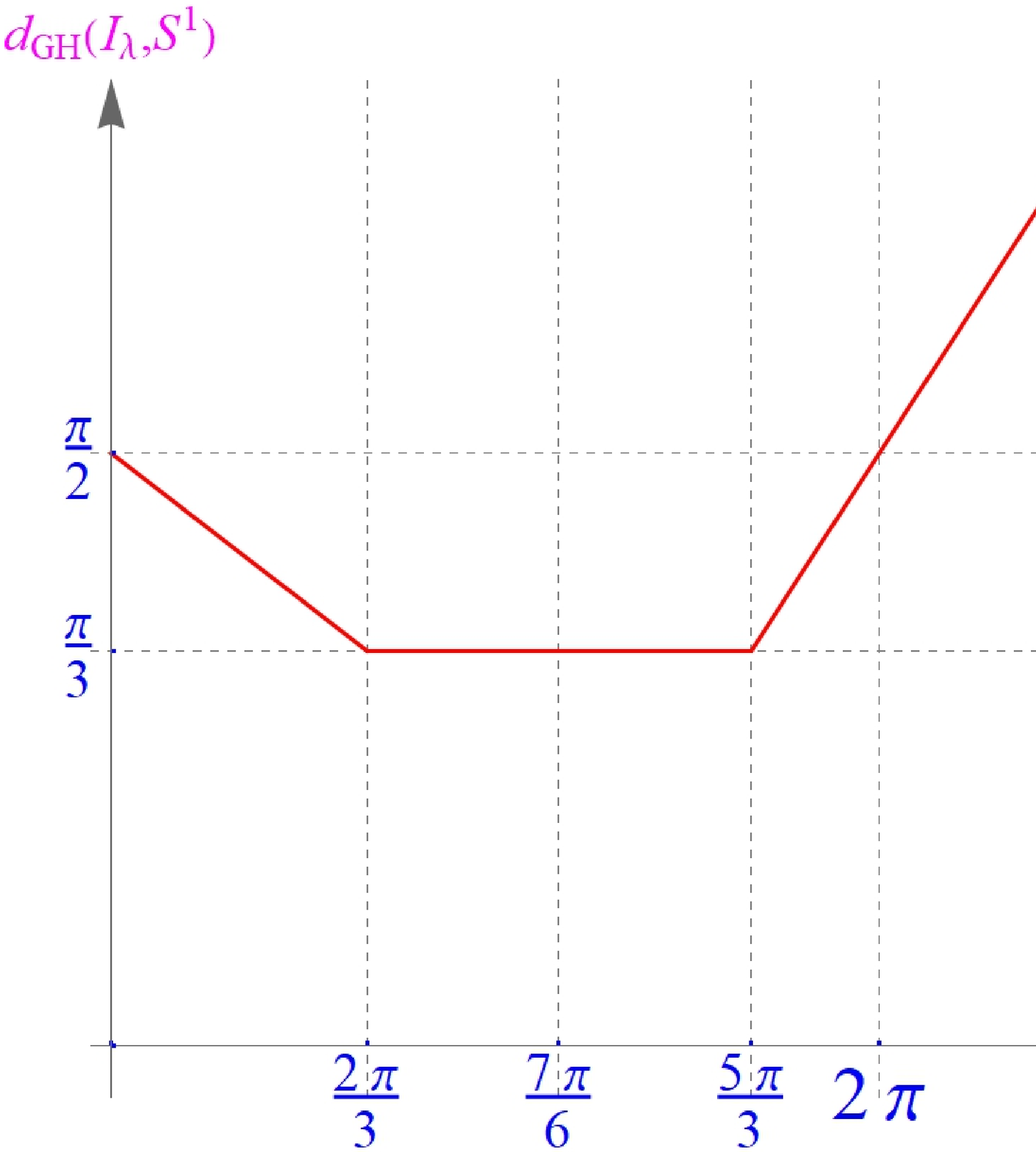}{0.4}{fig:figure}{Gromov--Hausdorff distance between circle and segment.}


\vfill\eject

\end{document}